\documentclass[11pt]{article}

\usepackage[utf8]{inputenc}
\usepackage[T1]{fontenc}
\usepackage[english]{babel}

\usepackage[a4paper,margin=1in]{geometry}
\usepackage{lmodern}
\usepackage{microtype}

\usepackage{amsmath,amssymb,amsthm,mathtools}

\usepackage{hyperref}
\hypersetup{
  colorlinks = true,
  linkcolor  = blue,
  citecolor  = blue,
  urlcolor   = blue
}

\numberwithin{equation}{section}

\theoremstyle{plain}
\newtheorem{theorem}{Theorem}[section]
\newtheorem{proposition}[theorem]{Proposition}

\theoremstyle{definition}
\newtheorem{definition}[theorem]{Definition}
\newtheorem{example}[theorem]{Example}
\newtheorem{remark}[theorem]{Remark}

\newcommand{\R}{\mathbb{R}}
\newcommand{\N}{\mathbb{N}}
\newcommand{\norm}[1]{\left\lVert #1 \right\rVert}
\newcommand{\ip}[2]{\left\langle #1,#2\right\rangle}

\title{A Beginner-Friendly Note on Maximal Monotone Operators}

\author{Hikmatullo Ismatov\\[0.3em]
\small King Abdullah University of Science and Technology (KAUST)}
\date{30/11/2025}

\begin{document}

\maketitle
\begin{abstract}
We give a self-contained and introductory account of some basic functional analytic tools needed to understand maximal monotone operators in Hilbert spaces. We review domains of (possibly unbounded) operators, closed sets and closed operators, and provide concrete examples of bounded and unbounded operators in both finite and infinite dimensions. We then explain in detail a fundamental result of Br\'ezis: if $A$ is a maximal monotone linear operator, then its domain is dense, $A$ is closed, and $(I+\lambda A)^{-1}$ is a non-expansive mapping for every $\lambda>0$. The Banach fixed point theorem (contraction mapping principle) is stated and used as a key ingredient in the analysis. The presentation is aimed at beginning graduate students and readers seeing these notions for the first time.
\end{abstract}



\tableofcontents

\section{Introduction}

The aim of these notes is to give a first contact with \emph{maximal monotone operators} in a real Hilbert space $H$. This topic appears very early in the study of:
\begin{itemize}
  \item linear and nonlinear evolution equations (heat equation, parabolic PDEs),
  \item variational inequalities and gradient flows,
  \item Mean Field Games and other PDE models with monotone structure,
  \item the Hille--Yosida theorem and semigroup theory.
\end{itemize}

Even in the purely linear case, the basic properties of a maximal monotone operator $A$ are already quite rich:
\begin{itemize}
  \item the domain $D(A)$ is dense in $H$;
  \item $A$ is a closed (in general unbounded) operator;
  \item for every $\lambda>0$, the operator $I+\lambda A$ is bijective and its inverse is a contraction:
        \[
          \|(I+\lambda A)^{-1}\| \le 1.
        \]
\end{itemize}
These are the key ingredients behind the abstract theory of (linear) contraction semigroups, and they also serve as a model for the fully nonlinear theory of monotone operators in Banach spaces.

The structure of the notes is as follows:
\begin{itemize}
  \item In Section~\ref{sec:hilbert} we review basic Hilbert space notions: closed sets, orthogonal complements, and density.
  \item In Section~\ref{sec:operators} we discuss domains, graphs, boundedness and closedness of operators, with concrete examples (including unbounded operators).
  \item In Section~\ref{sec:monotone} we define monotone and maximal monotone linear operators.
  \item In Section~\ref{sec:contraction} we present the contraction mapping (Banach fixed point) theorem.
  \item In Section~\ref{sec:maximal-main} we prove the three basic properties of a maximal monotone operator $A$ following Br\'ezis.
  \item In Section~\ref{sec:examples-applications} we collect some simple examples and a PDE application, to indicate how these abstract objects arise in practice.
\end{itemize}
The level is deliberately elementary: we work only in Hilbert spaces and focus on the linear case, but the ideas already capture the core mechanisms used in more advanced settings.

\section{Hilbert spaces, closed sets and orthogonal complements}\label{sec:hilbert}

Let $(H,\ip{\cdot}{\cdot})$ be a real Hilbert space, with norm
$\norm{u} = \sqrt{\ip{u}{u}}$.

\subsection{Closed sets}

\begin{definition}[Closed set]
A subset $C\subset H$ is called \emph{closed} if whenever
$(x_n)_{n\in\N}$ is a sequence in $C$ with $x_n\to x$ in $H$, then
$x\in C$.
\end{definition}

Equivalently, $C$ is closed if $C$ contains all its limit points, or
if $C$ equals its closure $\overline{C}$.

\begin{example}
In $\R$ with the usual norm:
\begin{itemize}
  \item The interval $[0,1]$ is closed.
  \item The open interval $(0,1)$ is not closed, because
        $x_n = 1/n \in(0,1)$ converges to $0\notin(0,1)$.
\end{itemize}
\end{example}

\subsection{Orthogonal complements and density}

For a subspace $M\subset H$ we define its orthogonal complement by
\[
M^\perp = \{\,f\in H : \ip{f}{v}=0 \text{ for all } v\in M\,\}.
\]

\begin{proposition}[Density and orthogonal complement]
Let $M\subset H$ be a subspace.
Then the following are equivalent:
\begin{enumerate}
  \item $\overline{M} = H$, i.e.\ $M$ is dense in $H$;
  \item $M^\perp = \{0\}$.
\end{enumerate}
\end{proposition}

\begin{proof}
$(1)\Rightarrow(2)$:
Let $f\in H$ be such that $\ip{f}{v}=0$ for all $v\in M$.
Since $M$ is dense, for every $w\in H$ there exists a sequence
$v_n\in M$ such that $v_n\to w$.
By continuity of the inner product,
\[
\ip{f}{w} = \lim_{n\to\infty} \ip{f}{v_n} = 0.
\]
Thus $f$ is orthogonal to every $w\in H$, in particular
$\ip{f}{f}=0$, so $\norm{f}^2=0$ and $f=0$.

$(2)\Rightarrow(1)$:
Suppose $\overline{M}\neq H$.
Then there exists $x\in H\setminus\overline{M}$.
Since $\overline{M}$ is a closed subspace, we have an orthogonal
decomposition $H = \overline{M} \oplus \overline{M}^\perp$, so we can
write $x = y+z$ with $y\in\overline{M}$, $z\in\overline{M}^\perp$.
Because $x\notin \overline{M}$, we must have $z\neq 0$.
But $M\subset\overline{M}$ implies $\overline{M}^\perp \subset M^\perp$,
so $z\in M^\perp$ and $z\neq 0$, contradicting $M^\perp = \{0\}$.
Therefore $\overline{M}=H$.
\end{proof}

In particular, to prove that a domain $D(A)$ is dense in $H$ it is
enough to show that the only $f\in H$ orthogonal to all $v\in D(A)$
is $f=0$.

\section{Operators, domains, boundedness and closedness}\label{sec:operators}

\subsection{Domains and graphs}

Let $A$ be a (possibly nonlinear) operator from a subset of $H$ to $H$.

\begin{definition}[Domain and graph]
The \emph{domain} of $A$ is
\[
D(A) = \{\,u\in H : Au \text{ is defined}\,\}.
\]
The \emph{graph} of $A$ is
\[
\mathcal{G}(A) = \{(u,Au) : u\in D(A)\} \subset H\times H.
\]
\end{definition}

If $A$ is a bounded linear operator on $H$ in the usual functional
analysis sense, then by definition $D(A)=H$ and people often do not
write the domain explicitly.
For unbounded or nonlinear operators the domain is essential.

\subsection{Bounded and unbounded linear operators}

Let $X,Y$ be normed spaces and $T:X\to Y$ linear.

\begin{definition}[Bounded linear operator]
We say $T$ is \emph{bounded} if there exists $C>0$ such that
\[
\norm{Tx}_Y \le C \norm{x}_X \quad\text{for all }x\in X.
\]
Equivalently, $T$ is continuous.
\end{definition}

If no such $C$ exists, $T$ is called \emph{unbounded}.

\begin{example}[Finite dimension: all linear operators are bounded]
Let $X=Y=\R^n$ with any norm.
Every linear map $T:\R^n\to\R^n$ is bounded.
Indeed, the unit sphere is compact and $x\mapsto\norm{Tx}$ is
continuous, so $\sup_{\norm{x}=1} \norm{Tx} <\infty$.
\end{example}

Thus truly unbounded linear operators appear only in infinite
dimensional spaces.

\begin{example}[Bounded operators in infinite dimension]\label{ex:bounded-operators}
\quad
\begin{itemize}
  \item Identity on $\ell^2$: $I(x)=x$, $\norm{Ix}=\norm{x}$.
  \item Right shift on $\ell^2$:
        $S(x_1,x_2,\dots)=(0,x_1,x_2,\dots)$.
        Then $\norm{Sx}=\norm{x}$, so $\norm{S}=1$.
  \item Multiplication by a bounded function:
        let $H=L^2(0,1)$ and $g\in L^\infty(0,1)$, define
        $(Tf)(x) = g(x)f(x)$.
        Then
        \[
        \norm{Tf}_{L^2}^2
          = \int_0^1 |g(x)|^2 |f(x)|^2\,dx
          \le \norm{g}_{L^\infty}^2 \norm{f}_{L^2}^2.
        \]
        Thus $T$ is bounded with $\norm{T}\le\norm{g}_{L^\infty}$.
\end{itemize}
\end{example}

\begin{example}[Unbounded linear operators in infinite dimension]\label{ex:unbounded-operators}
\quad
\begin{itemize}
  \item \emph{Derivative on $L^2(0,1)$:}
        let $H=L^2(0,1)$ and define
        \[
        D: D(D)\subset H \to H, \quad Df = f',
        \]
        with $D(D)=H_0^1(0,1)$.
        The operator $D$ is linear and densely defined, but unbounded:
        there exist $f_n$ with $\norm{f_n}_{L^2}\to 0$ and
        $\norm{f_n'}_{L^2}\to \infty$.
  \item \emph{Diagonal operator on $\ell^2$:}
        let $H=\ell^2$ and define
        \[
        (Tx)_n = n x_n,\quad
        D(T) = \bigl\{x\in\ell^2 : \sum_{n\ge1} n^2 |x_n|^2 < \infty\bigr\}.
        \]
        Then $T$ is linear and densely defined, but not bounded:
        for the unit vector $e_n$ with $1$ in position $n$ we have
        $\norm{e_n}=1$ yet $\norm{Te_n}=n\to\infty$.
\end{itemize}
\end{example}

\subsection{Closed graph and closed operators}

\begin{definition}[Closed graph]
Let $A:D(A)\subset H\to H$.
We say that $A$ has a \emph{closed graph} if $\mathcal{G}(A)$ is a
closed subset of $H\times H$, i.e.\ whenever
$u_n\in D(A)$, $u_n\to u$ in $H$, and $Au_n\to f$ in $H$, then
$(u,f)\in\mathcal{G}(A)$, or equivalently
$u\in D(A)$ and $Au=f$.
\end{definition}

\begin{definition}[Closed operator]
A linear operator $A:D(A)\subset H\to H$ is called \emph{closed} if
its graph $\mathcal{G}(A)$ is closed in $H\times H$.
\end{definition}

\begin{example}[Closed unbounded operator]
The derivative $D:H_0^1(0,1)\to L^2(0,1)$ defined above is closed:
if $u_n\in H_0^1(0,1)$, $u_n\to u$ in $L^2$ and $u_n'\to f$ in $L^2$,
then $u\in H_0^1(0,1)$ and $u'=f$ in the distributional sense.
Hence $(u,f)$ lies in the graph of $D$.
\end{example}

\begin{example}[An operator which is not closed]
Let $B:C^1([0,1])\subset L^2(0,1)\to L^2(0,1)$ be the classical
derivative $Bu=u'$.
There exists $u\in H^1(0,1)\setminus C^1([0,1])$ and a sequence
$u_n\in C^\infty([0,1])$ such that
$u_n\to u$ and $u_n'\to u'$ in $L^2$.
Thus $(u_n,Bu_n)\to(u,u')$ in $L^2\times L^2$, but
$u\notin D(B)$, so $(u,u')$ is not in the graph.
Hence the graph of $B$ is not closed and $B$ is not a closed operator.
\end{example}

\section{Monotone and maximal monotone linear operators}\label{sec:monotone}

\subsection{Monotone operators}

\begin{definition}[Monotone linear operator]
A (possibly unbounded) linear operator $A:D(A)\subset H\to H$ is
called \emph{monotone} if
\[
\ip{Au - Av}{u-v} \ge 0 \quad\text{for all }u,v\in D(A).
\]
If $A$ is linear, this is equivalent to
\[
\ip{Au}{u} \ge 0 \quad\text{for all }u\in D(A).
\]
\end{definition}

\begin{example}[Diagonal operator on $\ell^2$ is monotone]
Let $T$ be as in Example~\ref{ex:unbounded-operators}.
For $x\in D(T)$ we have
\[
\ip{Tx}{x}
  = \sum_{n=1}^\infty (Tx)_n \overline{x_n}
  = \sum_{n=1}^\infty n |x_n|^2 \ge 0.
\]
Thus $T$ is monotone.
\end{example}

\subsection{Maximal monotone operators}

Intuitively, a monotone operator is maximal if its graph cannot be
strictly enlarged while preserving monotonicity.

In this chapter Br\'ezis uses the following equivalent definition
for \emph{linear} operators.

\begin{definition}[Maximal monotone linear operator]
A monotone linear operator $A:D(A)\subset H\to H$ is called
\emph{maximal monotone} if, in addition,
\[
R(I+A) = H,
\]
i.e.\ for every $f\in H$ there exists $u\in D(A)$ such that
\[
u + Au = f.
\]
\end{definition}

This condition says that the operator $I+A$ is surjective.
It turns out to be equivalent (in the linear case) to the usual
``maximality of the graph'' definition.

\section{The contraction mapping theorem}\label{sec:contraction}

We recall the Banach fixed point theorem, which plays an important
role in the proof that $R(I+\lambda A)=H$ for all $\lambda>0$.

\begin{definition}[Contraction]
Let $(X,d)$ be a metric space.
A map $T:X\to X$ is called a \emph{contraction} if there exists a
constant $k\in[0,1)$ such that
\[
d(Tx,Ty) \le k\,d(x,y) \quad\text{for all }x,y\in X.
\]
In a normed space we write this as
$\norm{T(x)-T(y)}\le k\norm{x-y}$.
\end{definition}

\begin{theorem}[Contraction mapping / Banach fixed point]\label{thm:banach}
Let $(X,d)$ be a complete metric space and let $T:X\to X$ be a
contraction.
Then:
\begin{enumerate}
  \item There exists a unique fixed point $x^\ast\in X$ such that
        $T(x^\ast)=x^\ast$.
  \item For any starting point $x_0\in X$, the iterative sequence
        defined by $x_{n+1} = T(x_n)$ converges to $x^\ast$.
\end{enumerate}
\end{theorem}

\begin{proof}[Sketch of proof]
Fix $x_0\in X$ and define $x_{n+1}=T(x_n)$.
Then
\[
d(x_{n+1},x_n) = d(Tx_n, Tx_{n-1}) \le k\, d(x_n,x_{n-1}).
\]
By induction,
$d(x_{n+1},x_n) \le k^n d(x_1,x_0)$.
For $m>n$ we estimate
\[
d(x_m,x_n)
 \le \sum_{j=n}^{m-1} d(x_{j+1},x_j)
 \le d(x_1,x_0) \sum_{j=n}^{m-1} k^j
 \le \frac{k^n}{1-k} d(x_1,x_0).
\]
As $n\to\infty$ the right-hand side goes to $0$, so $(x_n)$ is a
Cauchy sequence.
Since $X$ is complete, $x_n\to x^\ast$ for some $x^\ast\in X$.
By continuity of $T$ (it is Lipschitz) we have
$T(x^\ast) = \lim T(x_n) = \lim x_{n+1} = x^\ast$, so $x^\ast$ is a
fixed point.

For uniqueness, suppose also $y^\ast$ is a fixed point.
Then
\[
d(x^\ast,y^\ast)
 = d(Tx^\ast, Ty^\ast)
 \le k\,d(x^\ast,y^\ast).
\]
Since $k<1$, this implies $d(x^\ast,y^\ast)=0$, so $x^\ast=y^\ast$.
\end{proof}

\section{A key result on maximal monotone operators}\label{sec:maximal-main}

We now explain in detail the three properties proved in Br\'ezis'
Proposition 7.1 for a maximal monotone linear operator $A$.

\begin{proposition}\label{prop:brezis-7.1}
Let $A:D(A)\subset H\to H$ be a maximal monotone linear operator.
Then
\begin{enumerate}
  \item[\textnormal{(a)}] $D(A)$ is dense in $H$;
  \item[\textnormal{(b)}] $A$ is a closed operator;
  \item[\textnormal{(c)}] For every $\lambda>0$ the operator
        $I+\lambda A:D(A)\to H$ is bijective,
        its inverse $(I+\lambda A)^{-1}$ is bounded, and
        $\|(I+\lambda A)^{-1}\|\le 1$.
\end{enumerate}
\end{proposition}

\subsection{Part (a): the domain is dense}

\begin{proof}[Proof of (a)]
Let $f\in H$ be such that $\ip{f}{v}=0$ for all $v\in D(A)$.
We claim that $f=0$.

Because $A$ is maximal monotone, by definition
$R(I+A)=H$, so there exists $v_0\in D(A)$ such that
\[
v_0 + Av_0 = f.
\]
Taking the scalar product with $v_0$ gives
\[
0 = \ip{f}{v_0}
  = \ip{v_0+Av_0}{v_0}
  = \norm{v_0}^2 + \ip{Av_0}{v_0}.
\]
Since $A$ is monotone, $\ip{Av_0}{v_0}\ge 0$, hence
$0 \ge \norm{v_0}^2$, which implies $v_0=0$.
Then $f=v_0+Av_0 = 0$.

We have shown that if $f$ is orthogonal to all $v\in D(A)$ then
$f=0$, i.e.\ $D(A)^\perp = \{0\}$.
By the Hilbert space result proved earlier, this is equivalent to
$\overline{D(A)}=H$, so $D(A)$ is dense in $H$.
\end{proof}

\subsection{Part (b): $A$ is closed}

\begin{proof}[Proof of (b)]
We first show that for every $f\in H$ there exists a unique
$u\in D(A)$ such that
\[
u + Au = f.
\]

\emph{Existence} follows from maximal monotonicity:
$R(I+A)=H$.

\emph{Uniqueness}:
Suppose $u,\bar u\in D(A)$ both satisfy
$u + Au = f$ and $\bar u + A\bar u = f$.
Subtracting gives
\[
(u-\bar u) + A(u-\bar u) = 0.
\]
Taking the inner product with $u-\bar u$ yields
\[
\norm{u-\bar u}^2 + \ip{A(u-\bar u)}{u-\bar u} = 0.
\]
By monotonicity, $\ip{A(u-\bar u)}{u-\bar u}\ge 0$, so
$\norm{u-\bar u}^2\le 0$ and $u=\bar u$.

Thus the map $f\mapsto u$ defined by
\[
(I+A)u = f
\]
is well-defined and linear, and we denote it by
$(I+A)^{-1}:H\to H$.

Next we show $\norm{(I+A)^{-1}}\le 1$.
Let $f\in H$ and $u=(I+A)^{-1}f$, so $u+Au=f$.
Taking the inner product with $u$ we obtain
\[
\norm{u}^2 + \ip{Au}{u} = \ip{f}{u} \le \norm{f}\,\norm{u}.
\]
Since $\ip{Au}{u}\ge 0$ (monotonicity), we have
\[
\norm{u}^2 \le \norm{f}\,\norm{u}.
\]
If $u\neq 0$, dividing by $\norm{u}$ gives $\norm{u}\le\norm{f}$;
if $u=0$ the inequality is trivial.
So in all cases
\[
\norm{(I+A)^{-1}f} = \norm{u} \le \norm{f}.
\]
Thus $(I+A)^{-1}$ is a bounded linear operator with
$\norm{(I+A)^{-1}}\le 1$.

We now prove that $A$ is closed.
Let $(u_n)$ be a sequence in $D(A)$ such that
$u_n\to u$ in $H$ and $Au_n\to f$ in $H$.
We must prove that $u\in D(A)$ and $Au=f$.

Since both sequences converge, we have
\[
u_n + Au_n \longrightarrow u+f \quad\text{in }H.
\]
But for each $n$,
\[
u_n = (I+A)^{-1}(u_n + Au_n).
\]
Using the continuity of $(I+A)^{-1}$ we get
\[
u_n = (I+A)^{-1}(u_n + Au_n)
   \longrightarrow (I+A)^{-1}(u+f).
\]
On the other hand $u_n\to u$, so by uniqueness of limits
\[
u = (I+A)^{-1}(u+f).
\]
By definition of $(I+A)^{-1}$, this means that $u\in D(A)$ and
\[
(I+A)u = u+f.
\]
Thus $u+Au = u+f$ and therefore $Au=f$.
This shows that the graph of $A$ is closed and hence $A$ is a closed
operator.
\end{proof}

\subsection{Part (c): $(I+\lambda A)^{-1}$ exists for all $\lambda>0$}

\begin{proof}[Proof of (c)]
We have already shown (with $\lambda=1$) that $I+A$ is bijective and
$(I+A)^{-1}$ is bounded with norm at most $1$.
We now prove that the surjectivity of $I+\lambda_0 A$ for some
$\lambda_0>0$ implies the surjectivity of $I+\lambda A$ for all
$\lambda>\lambda_0/2$.

\medskip\noindent
\emph{Step 1: set up a fixed point problem.}
Assume $R(I+\lambda_0 A)=H$ for some $\lambda_0>0$.
Then as in part (b), we know that for every $g\in H$ there is a unique
$u\in D(A)$ solving
\[
u + \lambda_0 Au = g,
\]
and the map $g\mapsto u$ defines a bounded linear operator
$(I+\lambda_0 A)^{-1}$ with norm at most $1$.

Let $\lambda>0$ and $f\in H$ be given.
We want to solve
\begin{equation}\label{eq:lambda-equation}
u + \lambda A u = f.
\end{equation}
We rewrite this in terms of $I+\lambda_0 A$.
Equation~\eqref{eq:lambda-equation} can be rewritten as
\[
u + \lambda_0 Au
  = \frac{\lambda_0}{\lambda} f
    + \Bigl(1-\frac{\lambda_0}{\lambda}\Bigr) u.
\]
Applying $(I+\lambda_0 A)^{-1}$ gives
\begin{equation}\label{eq:fixed-point}
u
  = (I+\lambda_0 A)^{-1}
    \Bigl[
      \frac{\lambda_0}{\lambda} f
      + \Bigl(1-\frac{\lambda_0}{\lambda}\Bigr) u
    \Bigr].
\end{equation}
Define $T:H\to H$ by
\[
T(u)
  = (I+\lambda_0 A)^{-1}
    \Bigl[
      \frac{\lambda_0}{\lambda} f
      + \Bigl(1-\frac{\lambda_0}{\lambda}\Bigr) u
    \Bigr].
\]
Then~\eqref{eq:fixed-point} is simply the fixed point equation
\[
u = T(u).
\]

\medskip\noindent
\emph{Step 2: $T$ is a contraction if $\lambda>\lambda_0/2$.}
For $u,v\in H$ we compute
\begin{align*}
\norm{T(u) - T(v)}
  &= \norm{
      (I+\lambda_0 A)^{-1}
      \Bigl(1-\frac{\lambda_0}{\lambda}\Bigr)(u-v)
    } \\
  &\le \norm{(I+\lambda_0 A)^{-1}}
       \Bigl|1-\frac{\lambda_0}{\lambda}\Bigr|
       \norm{u-v} \\
  &\le \Bigl|1-\frac{\lambda_0}{\lambda}\Bigr|
       \norm{u-v},
\end{align*}
since $\norm{(I+\lambda_0 A)^{-1}}\le 1$.
Thus $T$ is a contraction provided
\[
\Bigl|1-\frac{\lambda_0}{\lambda}\Bigr| < 1.
\]
This inequality is equivalent to $\lambda>\lambda_0/2$.

\medskip\noindent
\emph{Step 3: apply the contraction mapping theorem.}
If $\lambda>\lambda_0/2$, then $T$ is a contraction on the complete
space $H$.
By Theorem~\ref{thm:banach}, there exists a unique fixed point
$u\in H$ such that $T(u)=u$, i.e.\ $u$ satisfies~\eqref{eq:fixed-point}.
Tracing back the steps shows that $u\in D(A)$ and solves
$u+\lambda Au=f$.

Thus for every $\lambda>\lambda_0/2$ the operator $I+\lambda A$ is
surjective.
The estimate $\norm{(I+\lambda A)^{-1}}\le 1$ follows from the same
inner product argument as in part (b), replacing $A$ by $\lambda A$.

\medskip\noindent
\emph{Step 4: all positive $\lambda$ are covered by iteration.}
For a maximal monotone operator $A$ we know $R(I+A)=H$, i.e.\
$\lambda_0=1$ is admissible.
Therefore $R(I+\lambda A)=H$ for all $\lambda>1/2$.
Now repeat the argument with $\lambda_0=1/2$ to obtain surjectivity
for all $\lambda>1/4$, then with $\lambda_0=1/4$ to get surjectivity
for all $\lambda>1/8$, and so on.
After $n$ steps we know that $R(I+\lambda A)=H$ for all
$\lambda>2^{-n}$.
Given any $\lambda>0$, we can choose $n$ such that $2^{-n}<\lambda$,
and hence $I+\lambda A$ is surjective.
This proves that $R(I+\lambda A)=H$ for every $\lambda>0$.
\end{proof}

\begin{remark}
Part (c) also shows that $(I+\lambda A)^{-1}$ is nonexpansive:
$\norm{(I+\lambda A)^{-1}} \le 1$ for all $\lambda>0$.
In the nonlinear (maximal monotone) case this family plays the role of
the resolvent and is fundamental in the theory of nonlinear semigroups.
\end{remark}

\section{Examples and applications}\label{sec:examples-applications}

In this final section we illustrate the abstract theory with a few simple examples and one basic PDE application.

\subsection{A finite-dimensional example}

Let $H=\R^n$ with the standard inner product and let $A$ be a symmetric positive semidefinite matrix, i.e.
\[
A=A^\top,\qquad x^\top A x \ge 0 \quad\text{for all }x\in\R^n.
\]
Then:
\begin{itemize}
  \item $A$ is a bounded linear operator with $D(A)=H$;
  \item $A$ is monotone, because
        \[
          \ip{Ax-Ay}{x-y}
          = (x-y)^\top A (x-y) \ge 0;
        \]
  \item since $D(A)=H$ and we are in finite dimension, $A$ is automatically closed.
\end{itemize}
The operator $I+\lambda A$ is invertible for all $\lambda>0$, and its inverse is bounded. In fact, its norm is $\le 1$ if and only if all eigenvalues $\mu$ of $A$ satisfy $\mu\ge 0$, which is exactly positive semidefiniteness.

This trivial example shows that in finite dimension, maximal monotone linear operators are essentially the same as symmetric positive semidefinite matrices (possibly after a suitable choice of inner product).

\subsection{The derivative with Dirichlet boundary conditions}

Let $\Omega=(0,1)$ and $H=L^2(0,1)$. Consider
\[
A u = -\frac{d^2 u}{dx^2}
\]
with domain
\[
D(A) = H^2(0,1)\cap H_0^1(0,1),
\]
i.e.\ functions which are twice weakly differentiable, vanish at the boundary, and whose second derivative lies in $L^2(0,1)$.

Then:
\begin{itemize}
  \item $A$ is densely defined;
  \item $A$ is symmetric and
        \[
          \ip{Au}{u}
          = \int_0^1 -u''(x)\,u(x)\,dx
          = \int_0^1 |u'(x)|^2\,dx \ge 0,
        \]
        so $A$ is monotone;
  \item $A$ is closed (this follows from elliptic regularity and standard Sobolev space theory);
  \item $A$ is in fact maximal monotone, and $-A$ generates the heat semigroup on $L^2(0,1)$.
\end{itemize}

In this setting, solving
\[
u + \lambda A u = f
\]
for a given $f\in L^2(0,1)$ and $\lambda>0$ is the same as solving the boundary value problem
\[
\begin{cases}
u(x) - \lambda u''(x) = f(x), & x\in(0,1),\\
u(0)=u(1)=0.
\end{cases}
\]
The abstract theory tells us that this problem has a unique weak solution $u\in D(A)$ and that the solution operator
\[
(I+\lambda A)^{-1}:L^2(0,1)\to L^2(0,1)
\]
is a contraction.

\subsection{A resolvent as an elliptic PDE}

More generally, let $\Omega\subset\R^d$ be a bounded domain and consider the operator
\[
Au = -\Delta u
\]
with homogeneous Dirichlet boundary conditions and domain
\[
D(A) = H^2(\Omega)\cap H_0^1(\Omega).
\]
Then $A$ is maximal monotone on $L^2(\Omega)$.

For $f\in L^2(\Omega)$ and $\lambda>0$, the abstract resolvent equation
\[
u + \lambda A u = f
\]
is equivalent to the elliptic problem
\[
\begin{cases}
u(x) - \lambda \Delta u(x) = f(x), & x\in\Omega,\\
u = 0 & \text{on }\partial\Omega.
\end{cases}
\]
The maximal monotone theory guarantees a unique weak solution in $H_0^1(\Omega)$ and provides a clean functional analytic framework to study its dependence on $f$ and $\lambda$.

\subsection{Outlook}

The linear theory presented here extends in several directions:
\begin{itemize}
  \item to nonlinear maximal monotone operators in Banach spaces, where the resolvent $(I+\lambda A)^{-1}$ and its properties play a central role;
  \item to abstract evolution equations of the form $u'(t)+Au(t)=0$, whose solutions can be constructed using the resolvent and a nonlinear semigroup generated by $-A$;
  \item to applications in PDE, including nonlinear diffusion, variational inequalities, Mean Field Games, and Hamilton--Jacobi equations with convex Hamiltonians.
\end{itemize}
Even in these more advanced settings, the basic mechanisms---monotonicity, closedness, density of the domain, and the contraction mapping theorem---are essentially the same as in the linear Hilbert space case treated in these notes.

\bigskip

\noindent\textbf{Summary.}
In this note we have:
\begin{itemize}
  \item recalled basic notions about closed sets, orthogonal
        complements and density in a Hilbert space;
  \item discussed domains, graphs, boundedness and closedness of
        operators, with concrete examples;
  \item defined monotone and maximal monotone linear operators;
  \item stated and proved the contraction mapping theorem; and
  \item given detailed proofs of the three key properties in
        Proposition~\ref{prop:brezis-7.1}, together with simple examples and a PDE application.
\end{itemize}
These tools form the linear functional analytic backbone of the
Hille--Yosida theorem and are also a model for the nonlinear monotone
operator theory used in PDE and Mean Field Games.

\end{document}